\documentclass[11pt, reqno]{amsart}
\usepackage{setspace}
\usepackage{amsmath, amsthm, amscd, amsfonts, amssymb, graphicx, color}
\usepackage[a4paper,left=1.5cm,right=1.5cm,top=1.45cm,bottom=1.45cm]{geometry}
\usepackage{mathrsfs,amssymb}
\usepackage[bookmarksnumbered, colorlinks, plainpages]{hyperref}
\usepackage{amsmath,amsthm,amssymb, amsfonts, amssymb, graphicx, color}
\usepackage[bookmarksnumbered,colorlinks,plainpages]{hyperref}
\DeclareMathOperator{\VarJ}{VarJ}
\DeclareMathOperator{\VarK}{VarK}
\def\authorsaddresses#1{\dedicatory{#1}}
\newtheorem{theorem}{Theorem}[section]

\newtheorem{proposition}[theorem]{Proposition}

\theoremstyle{definition}
\newtheorem{definition}[theorem]{Definition}
\newtheorem{example}[theorem]{Example}

\theoremstyle{remark}
\newtheorem{remark}[theorem]{Remark}
\numberwithin{equation}{section}

\begin{document}
\setcounter{page}{1}

%%%%%%%%%%%%%%%%%%%%%%%%%%%%%%%%%%%%%%%%%%%%%%%
%% Please do not remove the following statement.
%%%%%%%%%%%%%%%%%%%%%%%%%%%%%%%%%%%%%%%%%%%%%%%
%\noindent \textbf{{\footnotesize By submitting this paper to W-ODRDCE I \textcolor[rgb]{1.00,0.00,0.00}{confirm }that
%(i) I and any other co-author(s) are responsible for its content and
%its originality; (ii) any possible co-authors agreed to its submission to W-ODRDCE.}}\\[1.00in]
%%%%%%%%%%%%%%%%%%%%%%%%%%%%%%%%%%%%%%%%%%%%%%%%
%% ==>start from here
%% Start from here
%%%%%%%%%%%%%%%%%%%%%%%%%%%%%%%%%%%%%%%%%%%%%%%

%%%%%%%%%%%%%%%%%%%%%%%%%%%%%%%%%%%%%%%%%%%%%%%%%%%%%%%%%%%%%%%%%%%%%
% Insert title of your article. Note: \title[short title]{full title}
%%%%%%%%%%%%%%%%%%%%%%%%%%%%%%%%%%%%%%%%%%%%%%%%%%%%%%%%%%%%%%%%%%%%%
\title[]{Measures of inaccuracy based on varextropy }
%%%%%%%%%%%%%%%%%%%%%%%%%%%%%%%%%%%%%%
% Author's name must be inserted here
%%%%%%%%%%%%%%%%%%%%%%%%%%%%%%%%%%%%%%

\author[F. Goodarzi, S. Ghafouri]{Faranak Goodarzi$^1$, Somayeh Ghafouri$^2$}

%%%%%%%%%%%%%%%%%%%%%%%%
% Addresses
%%%%%%%%%%%%%%%%%%%%%%%%
\authorsaddresses{
$^1$ Department of Statistics, University of Kashan, Kashan, Iran.\\
f-goodarzi@kashanu.ac.ir\\
$^2$ Department of Mathematics, Faculty of Science, Arak University, Arak, 384817758, Iran.}
%%%%%%%%%%%%%%%%%%%%%%%%%%%%
% Subject class; see http://www.ams.org/mathscinet/msc/msc2010.html
%%%%%%%%%%%%%%%%%%%%%%%%%%%%

%%%%%%%%%%%%%%%%%%%%%%%%%%%%%%%%%%%%%%%%%%%%%%%%%%%%%%%%%%%%%%%%%%%%%%%%%%%%%%%%%%%%%%%%%%
% keywords; Note that the number of keywords must be at least 3 items and at most 5 items.
%%%%%%%%%%%%%%%%%%%%%%%%%%%%%%%%%%%%%%%%%%%%%%%%%%%%%%%%%%%%%%%%%%%%%%%%%%%%%%%%%%%%%%%%%%
%%%%%%%%%%%%%%%%%%%%%%%%%%%%%%%%%%%%%%%%%%%%
% Abstract  of your extended abstract
% Note: the abstract should be 200 words or less with no reference number therein.
%The speaker is responsible for the proper formatting his/her talk by using the style
%file of the booklet of abstracts.
%%%%%%%%%%%%%%%%%%%%%%%%%%%%%%%%%%%%%%%%%%%
\begin{abstract}

Recently, varextropy has been introduced as a new dispersion index and a measure of information.  
In this article, we derive the generating function of extropy and present its infinite series representation. 
Furthermore, we propose new variability measures: the inaccuracy and weighted inaccuracy measures between two random variables based on varextropy and  we investigate their properties. 
We also obtain lower bounds for the inaccuracy measure  and compare them with each other.   In addition, we introduce a discrimination measure based on varextropy and employ it both for comparing  probability distributions and for assessing the goodness-of-fit of distributions to data and model selection problems. The proposed measure is also compared with the
 dispersion index derived from the Kullback-Leibler divergence given in \cite{BalaB1}.

\end{abstract}

%%%%%%%%%%%%%%%%%%%%%%%
\maketitle
%%%%%%%%%%%%%%%%%%%%%%%

%%%%%%%%%%%%%%%%%%%%%%%%%
% The body of your extended abstract
%%%%%%%%%%%%%%%%%%%%%%%%%

%%%%%%%%%%%%%%%%%%%%%%%%%%%%%%%%%
%% Note that The number of
%% pages of the extended abstract should have 3-4 pages. Papers prepared
%% in less than 3 pages, more than 4 pages or out of the style of the conference will be
%% returned
%%%%%%%%%%%%%%%%%%%%%%%%%%%%%%%%%
\section{Introduction}
Information theory has played a foundational role in the development of communication theory and has significantly influenced various other disciplines, such as statistical physics, computer science, and statistical inference. Let $X$ be an absolutely continuous random variable with probability density function (PDF) $f(x)$. Then, the Shannon differential entropy \cite{Shannon} of $X$ is  given by 

\begin{eqnarray}\label{En}
H(X)
=-\int_{-\infty}^{+\infty}f(x)\log f(x)dx,
\end{eqnarray} 
where $-\log f(X)$ represents the information content of $X$.
Information and entropy are closely related: higher entropy implies greater uncertainty and more information, while lower entropy implies less uncertainty and less information. 

It is noteworthy that distinct random variables can possess identical Shannon entropy values; in such cases, further characterization can be achieved through the variability of the information content. In this context, the variance of the information content, known as the varentropy (or variance of entropy) of a random variable $X$  (see \cite{Frad}), is defined as follows
\begin{align}\label{VEn}
V(X)&=
E[(-\log f(X))^2]-[H(X)]^2\nonumber\\&=
\int_{-\infty}^{+\infty}f(x)[\log f(x)]^2dx-
\left[\int_{-\infty}^{+\infty}f(x)\log f(x)dx\right]^2.
\end{align}  
The notion of entropy is recently entwined with a complementary dual measure, designated as extropy, by \cite{Lad}. 
The extropy of the random variable $X$ is defined as:
\begin{align}
J(X)=E\left[-\frac{1}{2}f(X)\right]=-\frac{1}{2}\int_{-\infty}^{+\infty}f^2(x)dx.
\end{align}
Recently, a new measure of uncertainty was introduced in \cite{Vaselabadi}, which can be viewed as a complementary  
measure to varentropy and as an alternative to Shannon entropy. This measure, referred to as varextropy
is based on the probability density function and is defined as follows:
\begin{align}\label{eqve}
\VarJ(X)&=Var\left[-\frac{1}{2}f(X)\right]=\frac{1}{4}E(f^2(X))-J^2(X).
\end{align}   
Let $X$ and $Y$ denote two non-negative continuous random variables with probability density functions $f$ and $g$, respectively. 
Balakrishnan \cite{BalaB} introduced the concept of weighted extropy as
\begin{align}\label{1e}
J^w(X)=-\frac{1}{2}\int_{0}^{+\infty}xf^2(x)dx.
\end{align}
While extropy is shift-invariant, weighted extropy depends on the location parameter. Consequently, weighted extropy provides greater flexibility compared to extropy.

Weighted varextropy also defined by \cite{chacru} as follows  
\begin{align}\label{wvar1}
\VarJ^w(X)=\frac{1}{4}\left[\int_{0}^{+\infty}x^2f^3(x)dx\right]-(J^w(X))^2.
\end{align}
The weighted varextropy incorporates the magnitude of the random variable through linear weighting, enabling it to account for both the probabilities and quantitative characteristics of events, and thereby providing a more informative measure of information variability. 
Moreover, classical varextropy is invariant under translations and exhibits a simple behavior under scaling, whereas the weighted varextropy is generally not affine invariant and is more sensitive to changes in scale and location. 

Hashempour et al. \cite{Hashemm} had defined 
the measure of inaccuracy associated with two random variables $X$ and $Y$ as

\begin{align}\label{exsh}
J(X, Y)=J(X)+J(X|Y)=-\frac{1}{2}\int_{0}^{+\infty}f(x)g(x)dx,
\end{align}
where $J(X|Y)=\frac{1}{2}\int_{0}^{+\infty}f(x)[f(x)-g(x)]dx$ is discrimination measure of variable $X$ with respect to variable $Y$ based on extropy and inaccuracy between density functions $f(x)$ and $g(x)$.
$J(X| Y)$ serves as an asymmetric measure of divergence between two distributions, implying that the order of the arguments is essential and, in general, $J(X|Y)\ne J(Y| X)$. 
%Moreover, it can be interpreted as the expected difference between the densities $f$ and $g$ with respect to $f$, providing a directional assessment of how $g$  deviates from $f$.

Furthermore, Hashempour and Kazemi \cite{Hashem} introduced the weighted discrimination measure of variable $X$ with respect to variable $Y$,  deﬁned as
\begin{align}\label{2e}
J^w(X| Y)=\frac{1}{2}\int_{0}^{+\infty}xf(x)[f(x)-g(x)]dx.
\end{align}
Adding \eqref{1e} and \eqref{2e}, we obtain weighted extropy inaccuracy (WJI) as
\begin{align}
J^w(X, Y)=J^w(X)+J^w(X| Y)=-\frac{1}{2}\int_{0}^{+\infty}xf(x)g(x)dx.
\end{align}
If $f(x)=g(x)$ then WJI reduces to Equation \eqref{1e}. 
%%%%%%%%%%%%%%%%%%%%%%%%%%%%%%%%%%%%%%%%%%%%%%%%%%%%%%%%%%%%%%%%%%%%%
َAs stated in \cite{Hashem},  WJI measures the cost of model misspecification by comparing the true density $f(x)$ with an assumed density $g(x)$, while incorporating the magnitude of each observation into its formulation. In this framework, the discrimination between $f$ and $g$ is influenced not only by their discrepancy but also by the contribution of individual observations, providing a more informative assessment of model differences. 
%%%%Section 2
%
%%%%%%%%%%%%%%%%%%%%%%%%%%%%%%%%%%%%%%%%%%%%%%%%%%%%%%%%%%%%%%%%%%%%%

The paper is organized as follows. In Section 2, we introduce the generating functions associated with extropy and present some preliminary results. 
In Section 3, a weighted inaccuracy measure based on varextropy is proposed and its properties under transformations of $X$ and $Y$ are studied. As a special case, $\VarJ(X,Y)$ is examined, and lower bounds for $\VarJ(X,Y)$ are derived.  In Section 4, a discrimination measure based on varextropy is developed. In Section 5, real lifetime data sets are analyzed to show the applicability of the proposed measures in goodness-of-fit assessment and model selection problems.  
\section{A generating function}\label{s2}
Balakrishnan et al. \cite{BalaB1} introduced a generating function for the Shannon entropy $H(X)$. In this section, we  present the definition of the generating function for extropy, followed by its representation as an infinite series.
We can define a generating function for the extropy $J(X)$ in the following form
\begin{align}
G_{J(X)}(t)=E_f[e^{t(-\frac{1}{2} f(X))}]=\int_{0}^{+\infty}e^{t(-\frac{1}{2} f(x))}f(x)dx.
\end{align}
Accordingly, an infinite series representation can be expressed as 
\begin{align}
G_{J(X)}(t)=E_f\left[1+t\left(-\frac{f(X)}{2}\right)+\frac{t^2}{2!}{\left(-\frac{f(X)}{2}\right)}^2+\frac{t^3}{3!}{\left(-\frac{f(X)}{2}\right)}^3+\cdots\right],
\end{align}
from which we readily find 
\begin{align}
\frac{d^k}{dt^k}G_{J(X)}(t)\Big|_{t=0}=E_f\Big[\Big(-\frac{1}{2}f(X)\Big)^k\Big].
\end{align}
In particular, we obtain $J(X)$ for $k = 1$, while for $k = 2$ the expression becomes $\VarJ(X) + (J(X))^2$. Similar to the role of $\VarJ(X)$ as a dispersion index for the extropy $J(X)$, higher-order derivatives of the generating function $G_{J(X)}(t)$ can be employed to define additional analogous measures associated with $J(X)$. For instance, a skewness measure can be expressed as
\begin{align*}
SkewJ(X)=E\left[\left(-\frac{1}{2}f(X)\right)^3\right]-3J(X)E\left[\left(\frac{1}{2}f(X)\right)^2\right]+2(J(X))^3.
\end{align*}
%and  the kurtosis associated with the extropy $J(X)$ is given by
%\begin{align*}
%KurtJ(X)=E\left[\left(\frac{1}{2}f(X)\right)^4\right]-4J(X)E\left[\left(\frac{1}{2}f(X)\right)^3\right]+6(J(X))^2E\left[\left(\frac{1}{2}f(X)\right)^2\right]-3(J(X))^4.
%\end{align*}
\begin{example}
Let us take $X\sim E(\lambda)$ with PDF $f(x)=\lambda e^{-\lambda x}, x>0, \lambda>0$. Then,
\[G_{J(X)}(t)=\int_{0}^{+\infty} e^{t(-\frac{1}{2}\lambda e^{-\lambda x})}\lambda e^{-\lambda x}dx=\frac{2}{t\lambda}\Big(1-\frac{1}{\sqrt{e^{t\lambda}}}\Big),\]
and by simple calculation we obtain $J(X)=-\frac{\lambda}{4}$ and $\VarJ(X)=\frac{\lambda^2}{48}$.
\end{example}
\section{Varjwinaccuracy}
In the study of uncertainty measures, various extensions of extropy and inaccuracy have been proposed to capture different aspects of variability and dispersion. In particular, varextropy provides a useful framework for quantifying the variability associated with information content.  

Motivated by the need to study the variability associated with weighted inaccuracy measures, we further investigate the previously defined weighted extropy inaccuracy $J^w(X, Y)$. While  $J^w(X, Y)$ describes the overall level of inaccuracy between the distributions $f$ and $g$, it
does not reflect the variability around this expected inaccuracy. 
To obtain additional insight into the stability and dispersion of the inaccuracy measures, we  introduce a new uncertainty measure denoted by $\VarJ^w(X, Y)$.  This measure provides additional information about fluctuations in the weighted inaccuracy and may therefore lead to  a more  comprehensive assessment of model misspecification.

Now, consider two non-negative continuous random variables $X$ and $Y$ with 
corresponding PDFs $f$ and $g$, defined on the  same support. 
We introduce a weighted measure of inaccuracy of $X$ with respect to $Y$, 
in terms of varextropy, referred to as the varjwinaccuracy, as follows.
\begin{definition}
Let $X$ and $Y$ be two non-negative continuous random variables with  probability density functions $f$ and $g$, respectively. Then, the varjwinaccuracy of $X$ and $Y$  is defined as
\begin{align}\label{3e}
\VarJ^w(X, Y)=Var_f^w\left[-\frac{1}{2}g(X)\right]=\frac{1}{4}\int_{0}^{+\infty}x^2g^2(x)f(x)dx-[J^w(X, Y)]^2.
\end{align}
\end{definition}
For identically distributed $X$ and $Y$,  the inaccuracy measure reduces to the weighted extropy, and thus the varjwinaccuracy reduces to the weighted varextropy given in \eqref{wvar1}.

In equation \eqref{3e}, if $w(x)=1$, then the measure of inaccuracy of $X$ about $Y$, in terms of varextropy (varjinaccuracy) is given by
\begin{align}\label{varj2}
\VarJ(X, Y)=Var_f\left[-\frac{1}{2}g(X)\right]=\frac{1}{4}\int_{0}^{+\infty}g^2(x)f(x)dx-[J(X, Y)]^2.
\end{align}
\begin{example}
Let $X\sim E(\lambda)$ and $Y\sim E(\eta)$. Then 
\begin{align*}
J^w(X, Y)=-\frac{1}{2}\int_{0}^{+\infty}x\eta e^{-\eta x}\lambda e^{-\lambda x}dx=-\frac{1}{2}\frac{\eta\lambda}{(\eta+\lambda)^2},
\end{align*}
and 
\begin{align*}
\VarJ^w(X, Y)=\frac{1}{4}\int_{0}^{+\infty}x^2\eta^2 e^{-2\eta x}\lambda e^{-\lambda x}dx-[J^w(X, Y)]^2
=\frac{1}{4}\frac{\eta^2\lambda(2\eta^4+2\eta\lambda^3+\lambda^4)}{(\eta+\lambda)^4(2\eta+\lambda)^3}.
\end{align*}
\end{example}
%%%%%%%%%%%%%%%%%%%%%%%%%%%%%%%%%%%%%%%%%%%%%%%%%%%%%%%%%%%
In the following theorem, we examine the effect of  strictly monotone  transformations on the varjwinaccuracy.
\begin{theorem}\label{th1}
 Let $X$ and $Y$ be random variables supported on a  common set $S$, with associated PDFs $f$ and $g$. For a strictly monotone transformation $\phi$,
define  $\tilde X=\phi (X)$ and  $\tilde Y=\phi (Y)$, whose  PDF’s are denoted by $\tilde f$ and $\tilde g$, respectively. Then, the following relation holds
\begin{align}
\VarJ^w(\tilde X, \tilde Y)=\frac{1}{4}Var_f\left[\frac{\phi(X)}{\phi^{\prime}(X)}g(X)\right].
\end{align}
\end{theorem}
\begin{proof}
Without loss of generality, let  $S = (0, +\infty)$, and such that $\phi$ strictly increasing from $\phi(0)$
to $+\infty$. Consequently, the common support of $\tilde X$ and $\tilde Y$ becomes $(\phi(0), +\infty)$ and invoking \eqref{3e}, we have 
\begin{align*}
\VarJ^w(\tilde X, \tilde Y)&=\frac{1}{4}\left[\int_{0}^{+\infty}x^2(\tilde g(x))^2\tilde f(x)dx-\left(\int_{0}^{+\infty}x\tilde g(x)\tilde f(x)dx\right)^2\right]\\&=\frac{1}{4}\int_{\phi(0)}^{+\infty}y^2\frac{1}{(\phi^{\prime}(\phi^{-1}(y)))^3}
f(\phi^{-1}(y))(g(\phi^{-1}(y)))^2dy\\&-\frac{1}{4}\left(\int_{\phi(0)}^{+\infty}y\frac{1}{(\phi^{\prime}(\phi^{-1}(y)))^2}f(\phi^{-1}(y))g(\phi^{-1}(y))dy\right)^2
\\&=\frac{1}{4}\left[\int_{0}^{+\infty}\left(\frac{\phi(z)}{\phi^{\prime}(z)}\right)^2g^2(z)f(z)dz-\left(\int_{0}^{+\infty}\frac{\phi(z)}{\phi^{\prime}(z)}g(z)f(z)dz\right)^2\right]
\\&=\frac{1}{4}
Var_f\left[\frac{\phi(X)}{\phi^{\prime}(X)}g(X)\right].
\end{align*}
\end{proof}
\begin{remark}
Let $X$ and $Y$  be two random variables with common support $S$ and PDFs $f$ and $g$, respectively. Let $a>0$ and $b\geq 0$ and the variables $\tilde X$ and $\tilde Y$ be $\tilde X=aX+b$ and $\tilde Y=aY+b$ with PDFs $\tilde f$ and  $\tilde g$, respectively. Then, by Theorem \ref{th1} and the properties of variance, we have
\begin{align}\label{1oc}
\VarJ^w(\tilde X, \tilde Y)&=\frac{1}{4}Var_f\left[\frac{aX+b}{a}g(X)\right]\nonumber\\&=\frac{1}{a^2}\Big\{a^2\VarJ^w(X, Y)+b^2\VarJ(X, Y)+2abCov_f(Xg(X), g(X))\Big\}.
\end{align}
Hence, the  varjwinaccuracy is not invariant  under affine transformations.
\end{remark}
We now investigate the conditions under which $\VarJ(X, Y )$ vanishes.
\begin{theorem}\label{thm2.6}
Suppose $X$ and $Y$ are random variables defined on the same support $S$, with PDFs $f$ and $g$, such that $g\in L^2(f)$. Then, $\VarJ(X, Y)=0$ if and only if $Y$ follows a uniform distribution over $S$.
\end{theorem}
\begin{proof}
If $Y$ has a uniform distribution, then it is easy to show that $\VarJ(X, Y)=0$. Conversely, if $\VarJ(X, Y)=Var_f[-\frac{1}{2}g(X)]=0$, then  $-\frac{1}{2}g(x)=c$ where $c$ is a constant. Hence, $Y$ has a uniform distribution.
\end{proof}
Now, our aim is to obtain a lower bound for varjinaccuracy. The variance $\VarJ(X, Y)$ captures the variability of the inaccuracy measure around its mean. Lower bounds for this quantity characterize the minimum attainable variability, and hence provide insight into the intrinsic stability of $J(X, Y)$ under the given distributions.   
%%%%%%%%%%%%%%%%%%%%%%%%%%%%%%%%%%%%%%%%%%%%%%%%%%%%%%%%%%%%%%%%%
\begin{remark}\label{23or}
Let us now recall the lower bound for the $Var[g(X)]$, provided in \cite{Afend}. Assume that $X$ satisfies  
\[\int_{-\infty}^{x}(\mu-t)f(t)dt=q(x)f(x),\hspace{0.5cm}x\in\mathbb{R},\]
and has finite moment of order $2n$ for some fixed $n\geq 1$, where $\mu=E[X]<\infty$. Then, for any function $g$  satisfying $E[q^k(X)|g^{(k)}(X)|] < \infty, k = 0,1, \ldots, n,$ we have the inequality 
\begin{align}\label{4.4}
Var[g(X)]\geq \sum_{k=1}^{n}\frac{E^2[q^{k}(X)g^{(k)}(X)]}{k!E[q^{k}(X)]\prod_{j=k-1}^{2k-2}(1-j\delta)},
\end{align}
with equality if and only if $g$ is a polynomial of degree at most $n$. (Note that $E|X|^{2n}<\infty$ implies $\delta<{(2n-1)}^{-1}$ and thus that  $\delta\notin\{1, \frac{1}{2}, \ldots, \frac{1}{2n-2}\}$  if $n\geq 2$).

Therefore, if $X$ and $Y$ are random variables defined in Theorem \ref{thm2.6} such that  the density function $g(x)$ is  $n$-times differentiable and satisfies the integrability conditions $E[q^k(X)|g^{(k)}(X)|] < \infty$ for $k = 0,1, \ldots, n$, then a lower bound for $\VarJ(X, Y)$ given by 
\begin{align}\label{Vj16}
\VarJ(X, Y)\geq \frac{1}{4}\sum_{k=1}^{n}\frac{E^2[q^{k}(X)g^{(k)}(X)]}{k!E[q^{k}(X)]\prod_{j=k-1}^{2k-2}(1-j\delta)}.
\end{align}
\end{remark}
Now in continuation, we obtain lower bound for $\VarJ(X,Y)$ for some distributions. 
\begin{example}\label{ex2.8}
If $X$ and $Y$ follow exponential distributions with parameters $\lambda$ and $\eta$, respectively, then 
as shown in \cite{Afend}, $q(x)=\frac{x}{\lambda}$ and $E[q^k(X)]=\frac{k!}{\lambda^{2k}}$. Hence, using \eqref{Vj16}, we obtain
\begin{align}\label{low1}
\VarJ(X, Y)&\geq\frac{1}{4}\sum_{k=1}^{n}\frac{1}{(k!)^2}\left(\int_{0}^{+\infty}x^k\eta^{k+1} e^{-\eta x}\lambda e^{-\lambda x}dx\right)^2
=\frac{1}{4}\sum_{k=1}^{n}\frac{\lambda^2\eta^{2k+2}}{(\lambda+\eta)^{2k+2}}\nonumber\\&=\frac{\lambda\eta^4\left[(\lambda+\eta)^{2n}-\eta^{2n}\right]}{4(\lambda+2\eta)(\lambda+\eta)^{2n+2}}=\frac{\lambda\eta^4}{4(\lambda+2\eta)(\lambda+\eta)^2}\left[1-\left(\frac{\eta}{\lambda+\eta}\right)^{2n}\right].
\end{align}
As $n$ increases, the lower bound becomes a more accurate approximation of the actual value of  $\VarJ(X, Y)=\frac{\lambda\eta^4}{4(\lambda+2\eta)(\lambda+\eta)^2}$.
\end{example}
\begin{example}
Let $X$  has a power distribution with PDF $f(x)=ax^{a-1}, 0<x<1, a>0$. Let $Y$ be a continuous random variable with support $(0, 1)$ and PDF $g(x)$. Then, since $q(x)=\frac{x(1-x)}{a+1}$, and $\delta=-\frac{1}{a+1}$, a lower bound for  $\VarJ(X, Y)$ is given as follows 
\begin{align*}
\VarJ(X, Y)\geq \frac{1}{4a}\sum_{k=1}^n\frac{a+2k}{(k!)^2}E^2[X^k({1-X)}^kg^{(k)}(X)].
\end{align*}
For instance, if $Y$ has PDF $g(x)=bx^{b-1}$ for $0<x<1$, where $b>0$, and $g$ is $n$-times differentiable then
\begin{align*}
\VarJ(X, Y)&\geq \frac{ab^2}{4}
\sum_{k=1}^n(a+2k){\left(\frac{\Gamma(b)\Gamma(a+b-1)}{\Gamma(b-k)\Gamma(a+b+k)}\right)}^2\nonumber
\\&=\frac{a(b\Gamma(b)\Gamma(a+b-1))^2}{4(a+2b-2)}\left[\left(\frac{a+b}{\Gamma(b-1)\Gamma(a+b+1)}\right)^2-\left(\frac{a+b+n}{\Gamma(b-n-1)\Gamma(n+a+b+1)}\right)^2\right],
\end{align*}
such that $b\geq n+1$. By Remark \ref{23or}, when  $b=n+1$ the equality condition holds. 
\end{example}
In the next theorem, we derive an additional lower bound for the varjinaccuracy by employing Chebyshev's inequality.
\begin{theorem}
Let $X$ and $Y$ be random variables supported on a common set $S$, with associated PDFs $f$ and $g$, and let $\epsilon>0$. Then,  the following lower bound for varjinaccuracy holds
\begin{align}
\VarJ(X, Y)\geq \epsilon^2\Big\{P\big[g(X)\geq 2(\epsilon-J(X, Y))\big]+P\big[g(X)\leq 2(-\epsilon-J(X, Y))\big]\Big\}.
\end{align} 
\end{theorem}
\begin{proof}
The proof follows immediately from \eqref{exsh}, \eqref{varj2}, and Chebyshev's inequality. 
%Based on the defitions of inaccuracy between $X$ and $Y$ based on extropy and varjinaccuracy given in \eqref{exsh} and \eqref{varj2} and yet chebyshev inequality yields
%\begin{align*}
%\VarJ(X, Y)&\geq\epsilon^2 P\left(\Big|-\frac{1}{2}g(X)-J(X, Y)\Big|\geq \epsilon\right)\nonumber\\
%&=\epsilon^2\left\{P\left(\frac{1}{2}g(X)+J(X, Y)\geq \epsilon\right)+P\left(\frac{1}{2}g(X)+J(X, Y)\leq -\epsilon\right)\right\}
%\nonumber\\&=\epsilon^2\Big\{P\Big(g(X)\geq 2(\epsilon-J(X, Y))\Big)+P\Big(g(X)\leq 2(-\epsilon-J(X, Y))\Big)\Big\}.
%\end{align*}
\end{proof}
\begin{remark}
In the above theorem if $g$ is strictly decreasing in $S$, then
\begin{align}
\VarJ(X, Y)&\geq\epsilon^2\Big\{F\Big(g^{-1} (2(\epsilon-J(X, Y)))\Big)+\overline F\Big(g^{-1} (2(-\epsilon-J(X, Y)))\Big)\Big\}.
\end{align}
\end{remark}
\begin{example}
In this example, we want to find a lower bound for $VarJ(X, Y)$ computed in example \ref{ex2.8}. Since $g^{-1}(z)=-\frac{1}{\eta}\log {\frac{z}{\eta}}$
and $J(X, Y)=-\frac{1}{2}\frac{\lambda\eta}{\lambda+\eta}$, hence 
\begin{align}\label{low2}
\VarJ(X, Y)\geq
          \epsilon^2\left(\left[-\frac{2\epsilon}{\eta}+\frac{\lambda}{\lambda+
         \eta}\right]^{\frac{\lambda}{\eta}}+1-\left[\frac{2\epsilon}{\eta}+\frac{\lambda}{\lambda+
          \eta}\right]^{\frac{\lambda}{\eta}}\right).
\end{align}
Notice that $\lambda$, $\eta$, and $\epsilon$ should be chosen  such  that the expressions in brackets are less than or equal to  1.

To compare the two lower bounds obtained for the varjinaccuracy in equations \eqref{low1} and \eqref{low2}, for instance, assume $\lambda=5$, $\eta=4$, $\epsilon=0.02$, and $n=5$. The corresponding values are $0.3038022477$ and $0.000408$, respectively. Comparing these with the true value of $\VarJ(X, Y)=0.3038936372$, it is evident that the first bound is extremely accurate. 
\end{example}
%\begin{proposition}
%Let $X_1, X_2, \ldots, X_n$ be a random sample from an absolutely continuous with probability density function $f(x)$ and  distribution function $F(x)$, and $X_{i:n}$ be  the ith order statistic, then 
%\begin{align}
%\VarJ(X_{i:n}, X)=\frac{1}{4}\left\{
%\frac{B(2i-1, 2n-2i+1)}{B^2(i, n-i+1)}E\left[f^2(F^{-1}(\mathscr{U}_2))\right]-E^2\left[f(F^{-1}(\mathscr{U}_1))\right]
%\right\},
%\end{align} 
%where $\mathscr{U}_1\sim Beta(i, n-i+1)$  and  $\mathscr{U}_2\sim Beta(2i-1, 2n-2i+1)$.
%\end{proposition}
%%%%%%%%%%%%%%%%%%%%%%%%%%%%%%%%%%%%%%%%%%%%%%%%%%%%%%%%%%%%%%%%%%%
%%%%%%%%%%%%%%%%%%%%%%%%%%%%%%%%%%%%%%%%%%%%%%%%%%%%%%%%%%%%%%%%%%%%
%%
%%
%    Order Sttatistics
%%
%%%%%%%%%%%%%%%%%%%%%%%%%%%%%%%%%%%%%%%%%%%%%%%%%%%%%%%%%%%%%%%%%%%%
%%%%%%%%%%%%%%%%%%%%%%%%%%%%%%%%%%%%%%%%%%%%%%%%%%%%%%%%%%%%%%%%%%%%

\section{A dispersion index} 
In this section, we introduce a dispersion index to capture the variability between the two density functions, based on the quantity  $\frac{1}{2}(f(x)-g(x))$, defined as follows.
\begin{definition}
If $X$ and $Y$ be two non-negative random variables with  probability density functions $f$ and $g$,
respectively. Then, a dispersion index of X and Y is defined as
\begin{align}
\VarJ(X| Y)&=Var\left[\frac{1}{2}(f(X)-g(X))\right]\nonumber\\&=E\left[\Big(\frac{1}{2}(f(X)-g(X))\Big)^2\right]-E^2\left[\frac{1}{2}(f(X)-g(X))\right]
\nonumber\\&=\frac{1}{4}\int_{0}^{+\infty}(f(x)-g(x))^2f(x)dx-(J(X|Y))^2.
\end{align}
\end{definition}
Under certain conditions on the underlying density functions, the proposed measure $VarJ(X| Y)$ may become infinite. In particular, this may occure when the density functions exhibit singular behavior, resulting in the non-integrability of the term $(f(x)-g(x))^2f(x)$ . For example, let $f(x)=\frac{1}{2\sqrt{x}}, 0<x<1,$ and $g(x)=\frac{1}{3x^{\frac{2}{3}}}, 0<x<1$.

In analogy with Section \ref{s2}, we can introduce a generating function for $J(X| Y)$ as
\begin{align}
G_{J(X| Y)}(t)&=E_f[e^{t(\frac{1}{2}(f(X)-g(X)))}]\nonumber\\&=E_f\left[1+\frac{1}{2}(f(X)-g(X))t+(\frac{1}{2}(f(X)-g(X)))^2\frac{t^2}{2!}+
(\frac{1}{2}(f(X)-g(X)))^3\frac{t^3}{3!}+\ldots\right],
\end{align}
which readily yields
\begin{align}
\frac{d^k}{dt^k}G_{J(X| Y)}(t)\Big|_{t=0}=E_f\left[\left(\frac{1}{2}(f(X)-g(X))\right)^k\right].
\end{align}
In particular, for $k=1$, we recover $J(X| Y)$; for $k=2$, we obtain the experssion 
$E\left[
{\Big(\frac{1}{2}(f(X)-g(X))\Big)}^2
\right]$ 
which corresponds to the variance $\VarJ(X| Y)+(J(X|Y))^2$. 
%Based on this, one can define higher-order measures such as skewness  and kurtosis denoted by $SkewJ(X| Y)$ and $KurtJ(X| Y)$ using the successive derivatives of the generating function $G_{J(X|Y)}(t)$, 
%following a similar procedure to that used in earlier section.

Analogous to the identity in \eqref{1oc}, the following proposition establishes an identity involving varextropy, varjinaccuracy and $\VarJ$ measures.
\begin{proposition}
Let $X$ and $Y$ be two non-negative random variables with common support S and PDF's $f$ and $g$, respectively. Then 
\[\VarJ(X| Y)=\frac{1}{4}Var[f(X)-g(X)]=\frac{1}{4}[\VarJ(X)+\VarJ(X, Y)-2Cov_f(f(X), g(X))].\]
\end{proposition}
%%%%%%%%%%%%%%%%%%%%%%%%%%%%%%%%%%%%%%%%%%%%%%%%%%%%%%%%%%%%%%%%%%%
\begin{proposition}\label{dm}
Suppose $X$ and $Y$ are random variables sharing a common support $S$, having probability density functions $f$ and $g$, respectively. If $f$ and $g$ are continuous on $S$, then, $\VarJ(X| Y)=0$ if and only if $X$ and $Y$ have identical distributions.
\end{proposition}
\begin{proof}
$\VarJ(X| Y)=Var\left[\frac{1}{2}(f(X)-g(X))\right]=0$ if and only if $\frac{1}{2}(f(X)-g(X))=c$. Now, since $\int_{S}f(x)dx=\int_{S}g(x)dx=1$, therefore we conclude that $c=0$. Hence $f$ and $g$ are identially distributed. 
\end{proof}
%%%%%%%%%%%%%%%%%%%%%%%%%%%%%%%%%%%%%%%%%%%%%%%%%%%%%%%%%%%%%%%%%%%%
\begin{remark}
According to Proposition \ref{dm}, we conclude that $\VarJ(X| Y)$ is a divergence measure. On the other hand, since $\VarJ(X| Y)$ is not equal to $\VarJ(Y| X)$,  the measure  is not symmetric.  Moreover, using a counterexample, we can show that $\VarJ(X|Y)$ does not satisfy the triangular inequality. Let 
$X$, $Y$, and $Z$ be exponentially distributed random variables with means $\frac{1}{2}$, $\frac{1}{3}$ and $\frac{100}{3}$, respectively. It can be easily shown that 
$\VarJ(X|Y)=0.028690$, $\VarJ(X|Z)=0.074528$, and $\VarJ(Z|Y)=0.010565$, hence $\VarJ(X|Z)>\VarJ(X|Y)+\VarJ(Z|Y)$, which violates the triangular inequality.
%\begin{proposition}
%Let $X_1, X_2, \ldots, X_n$ be a random sample from an absolutely continuous with probability density function $f(x)$ and  distribution function $F(x)$, and $X_{i:n}$ be  the ith order statistic, then 
%\begin{align}
%\VarJ(X_{i:n}|X)&=\nonumber\frac{1}{4}\left[\frac{B(3i-2, 3n-3i+1)}{B^3(i, n-i+1)}E\left[f^2(F^{-1}(\mathscr{U}_3))\right]-
%\frac{2B(2i-1, 2n-2i+1)}{B^2(i, n-i+1)}E\left[f^2(F^{-1}(\mathscr{U}_2))\right]\right.\\&\left.+ٍE[f^2(F^{-1}(\mathscr{U}_1))]
%-\left(\frac{B(2i-1, 2n-2i+1)}{B^2(i, n-i+1)}E[f(F^{-1}(\mathscr{U}_2))]-E[f(F^{-1}(\mathscr{U}_1))]\right)^2\right],
%\end{align}
%where $\mathscr{U}_3\sim Beta(3i-2, 3n-3i+1)$.
%\end{proposition}
\end{remark}
\section{Employing VarJ as a criterion in distributional goodness-of-fit testing}\label{section4}
$J(X| Y)$ is a dissimilarity measure %divergence quantifies the dissimilarity 
between two probability distributions. 
When the observed data are modeled by a random variable $X$ with probability density function $f$, it can be compared with different candidate distributions $Y$ (with density $g$) using $J(X| Y)$. A smaller value of $J(X| Y)$  indicates a closer match. However, when two distributions $Y_1$​ and $Y_2$ yield nearly equal dissimilarity measures, i.e., $J(X|Y_1)\approx J(X|Y_2)$ the choice between them may not be straightforward. In such situations, $\VarJ$ can be used as an auxiliary criterion, helping to select the model with greater stability (i.e., lower $\VarJ(X|Y)$).

To construct a criterion that incorporates both  $J(X| Y)$ and the related dispersion, we introduce a threshold $r$. If $J(X|Y_i)$ for $i=1, 2$, exceeds this threshold,  the corresponding distribution is not accepted. Assume that $J(X| Y_1)\leq J(X| Y_2)$, 
and $J(X|Y_2)<r$. As $J(X| Y_2)$  approaches $r$, it becomes more difficult to justify selecting $Y_2$ over $Y_1$. 

However, a higher dissimilarity may be tolerated  if it is compensated by a lower variance. In such cases, $\VarJ$ can be used  to normalize the difference between $r$ and the dissimilarity values, allowing for a more balanced  comparison. Accordingly, we prefer $Y_2$​ over $Y_1$ if the following inequality holds 
\begin{align}\label{17e}
\frac{r-J(X| Y_2)}{\sqrt{\VarJ(X|Y_2)}}>\frac{r-J(X| Y_1)}{\sqrt{\VarJ(X|Y_1)}}.
\end{align}

To implement the criterion in practical scenarios, it is necessary to define a value for the threshold $r$. While assigning an exact numerical value to $r$ may not always be feasible, it can instead be expressed in terms of the extropy divergences.  We set $r=2J(X| Y_1)$, under the assumption that $J(X| Y_1)\leq  J(X| Y_2)$. With this formulation, the decision rule presented earlier (inequality \eqref{17e}) can be rewritten as follows: we prefer $Y_2$ over $Y_1$ if the following inequality holds: 
\begin{align}
\frac{J(X|Y_1)}{\sqrt{\VarJ(X|Y_1)}}<\frac{2J(X|Y_1)-J(X|Y_2)}{\sqrt{\VarJ(X|Y_2)}},
\end{align}
which can be expressed as
\begin{align}\label{18e}
J(X|Y_2)<\left(2-\sqrt{\frac{\VarJ(X|Y_2)}{\VarJ(X|Y_1)}}\right)J(X|Y_1).
\end{align} 
\subsection{\bf Numerical Examples}\label{subsection4.1}
In the following examples, the inaccuracy measure $J(X| Y)$ and the proposed measure $VarJ(X| Y)$ are applied to real lifetime data sets in order to compare competing parametric models with nonparametric density estimates obtained via kernel density estimation (KDE). These 
examples illustrate the usefulness of the proposed measures in goodness-of-fit assessment and model selection problems involving lifetime distributions. 
In particular, the proposed criteria  provide direct and interpretable comparisons between competing fitted models and are also compared with standard goodness-of-fit and model selection procedures such as the Kolmogorov-Smirnov and Anderson-Darling tests, as well as AIC, BIC, and HQIC criteria.   
 %%%%%%%%%%%%%%%%%%%%%%%%%%%%%%%%%%%%%%%%%%%%%
%%%%%%%%%%%%%%%%%%%%%%%%%%%%%%%%%%%%%%%%%%%%%%%%%
\begin{example}\label{5.1}
The data (see \cite{Lawless}) represent the number of thousand miles at which different locomotive controls failed.
In a life test,  the failure times for the 37 failed units are
22.5, 37.5, 46.0, 48.5, 51.5, 53.0, 54.5, 57.5, 66.5, 68.0, 69.5, 76.5, 77.0, 78.5, 80.0,
81.5, 82.0, 83.0, 84.0, 91.5, 93.5, 102.5, 107.0, 108.5, 112.5, 113.5, 116.0, 117.0,
118.5, 119.0, 120.0, 122.5, 123.0, 127.5, 131.0, 132.5, 134.0.

Let $X$ denote the observed failure times with probability density function $f(x)$. We first estimate $f(x)$ using KDE with a Gaussian kernel. 
To assess whether a lognormal distribution provides an appropriate parametric model for the data, we estimate the parameters $\mu$ and $\sigma$ of the lognormal distribution. Using the maximum likelihood method, we obtain $(\hat{\mu}, \hat{\sigma})=(4.422567, 0.4031749)$. In addition, a Bayesian estimation procedure based on a Metropolis-Hastings  sampler is employed, yielding the posterior means $(\hat{\mu}_B, \hat{\sigma}_B)=(4.427955, 0.4516975)$. 

These parameter estimates are subsequently used for comparing the fitted lognormal models with the nonparametric KDE.
Let $Y_1\sim LG(4.422567, 0.4031749)$ and $Y_2\sim LG(4.427955, 0.4516975)$. 
The Kolmogorov-Smirnov p-values ($0.4814$ for MLE, $0.3832$ for Bayesian) and the Anderson-Darling p-values (0.3487 for MLE, 0.3349 for Bayesian) indicate
that both lognormal fits are accepted at the 5\% significance  level, but the MLE fit is preferable.
The estimated probability density functions of the data and the fitted distributions  $Y_1$ and $Y_2$ are shown in Figure \ref{fig1}.
\begin{figure}
\begin{center}
\includegraphics[width=7cm, height=7cm]{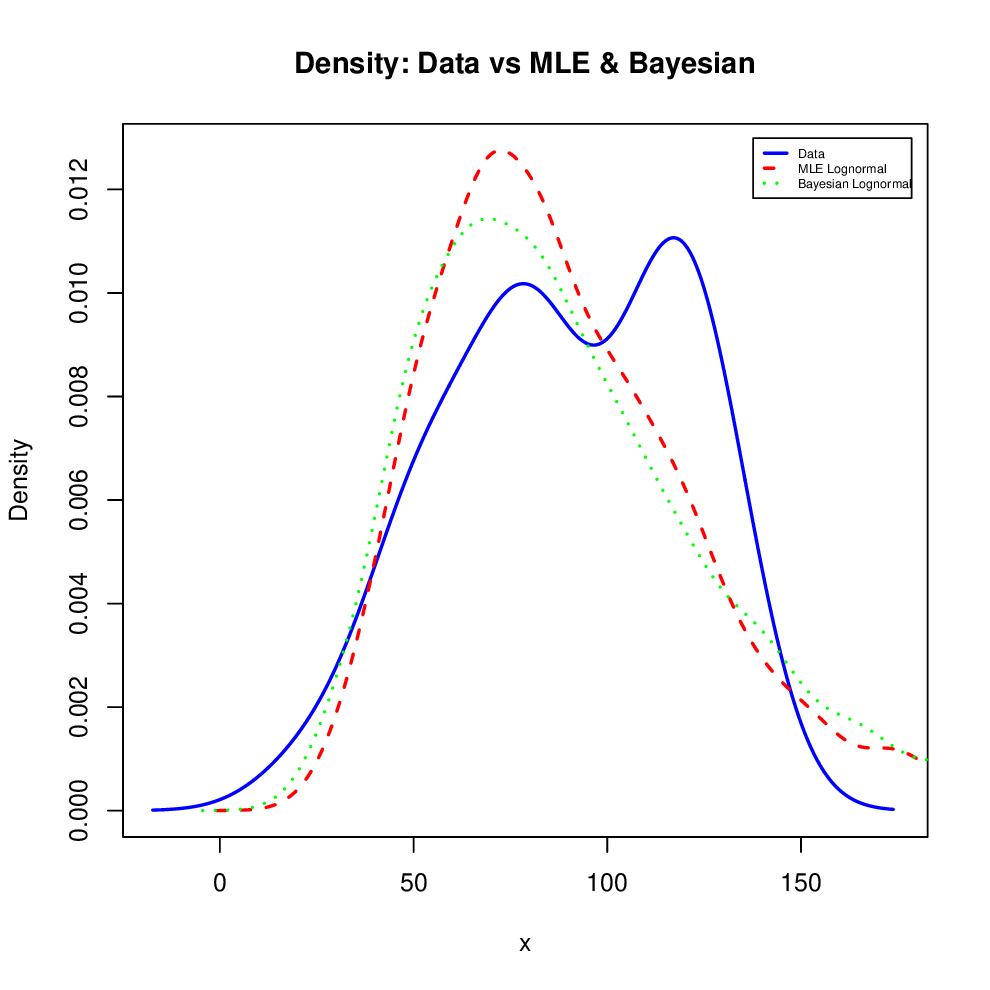}
\caption{\footnotesize{Plot of PDF's of $X$ , $Y_1$ and $Y_2$ for Example \ref{5.1}.}}\label{fig1}
\end{center}
\end{figure}

Using the definitions of the Kullback-Leibler divergence and the variance of inaccuracy (varinaccuracy) $\VarK$ given in equations (1) and (14) in \cite{BalaB1}, we  obtain:
\begin{align}
K(X, Y_1)&=0.12333932, \, \,  K(X, Y_2)=0.1396457,\nonumber \\ \VarK(X, Y_1)&=0.2358482, \, \, \VarK(X, Y_2)=0.1866185. 
\end{align}
In terms of discrimination measures,  we have:  
\begin{align}
J(X| Y_1)&=0.0001684737, \, \,  J(X| Y_2)=0.0003892534, \nonumber \\ \VarJ(X|Y_1)&=1.756753e^{-06}, \, \, \VarJ(X|Y_2)=1.756753e^{-06}. 
\end{align}
By comparing  the discrimination measures, we conclude that $Y_1$ is preferable to $Y_2$, since  $J(X|Y_1)\leq J(X|Y_2)$ and $\VarJ(X|Y_1)\leq \VarJ(X|Y_2)$. 
However, based on the Kullback-Leibler divergence  and its variability, since $K(X, Y_1)\leq K(X, Y_2)$ but $\VarK(X, Y_1)\geq \VarK(X, Y_2)$, determining which distribution provides a better fit requires using the following criterion:
\begin{align}\nonumber
K(X, Y_2)-\left(2-\sqrt{\frac{\VarK(X, Y_2)}{\VarK(X, Y_1)}}\right)K(X, Y_1)=0.0026213706.
\end{align} 
Since the resulting value is positive, this criterion also supports selecting $Y_1$.
Hence, the inaccuracy measure $J$ quickly and clearly identifies $Y_1$ as the better fit, without additional calculations. 

\end{example}
\begin{example}
We consider the aircraft window glass strength dataset documented by \cite{Alsh}. As an illustrative application, the inaccuracy  measure is employed to compare the fitted distributions and evaluate their relative performance. In the study by \cite{Alsh}, it was shown that among several one-parameter distributions, including the inverse Rayleigh and the A distributions, the A distribution provided the best fit with a Kolmogorov-Smirnov statistic of 0.393, and the MLE estimate of its parameter is $125.662$. Let $Y_1$ have the A distribution with parameter $125.662$.

In this example, we  investigate whether the A distribution  or the gamma distribution provides a better fit to the data.  The MLE estimators of  the shape and scale parameters of the gamma distribution are $18.925$ and $1.629$, respectively. Let $Y_2\sim G(18.925, 1.629)$.

The Kullback-Leibler divergences and varinaccuracies for the A and gamma distributions are given, respectively by:
\begin{align}
K(X, Y_1)&=0.20279863,\, \, K(X, Y_2)=0.03092508, 
\nonumber \\ \VarK(X, Y_1)&=0.37098035, \, \, \VarK(X, Y_2)=0.06289973.
\end{align}
Also, the discrimination measure and varjinaccuracy are:
\begin{align}
J(X| Y_1)&=-0.000261882, \, \,  J(X| Y_2)=-0.0008033541 , \nonumber \\ \VarJ(X|Y_1)&=5.237124e^{-05}, \, \, \VarJ(X|Y_2)=1.570272e^{-05}. 
\end{align}
Since the random variable $Y_2$ has a lower inaccuracy measure J and a smaller VarJ,   the gamma distribution is preferable 
according to these criteria, consistent with the assessment based  on the Kullback-Leibler divergence and its variability.
\end{example}
\begin{example}\label{exp1}
In this example, $23$ deep-groove ball bearings were tested, and
the number of revolutions until failure was recorded for each (see \cite{Lawless}, page 99). The observed failure times (in millions of revolutions) are: 17.88, 28.92, 33.00, 41.52, 42.12, 45.60, 48.40, 51.84, 51.96, 54.12, 55.56, 67.80, 68.64, 68.64, 68.88, 84.12, 93.12, 98.64, 105.12, 105.84, 127.92, 128.04 and 173.40. Although the bearings were inspected periodically,  the failure times are treated as continuous for analysis. 

The Weibull distribution with PDF $g_1(x)=\lambda \alpha (x\lambda)^{\alpha-1}\exp(-(\lambda x)^{\alpha}), x>0$ $(W(\alpha, \lambda)$, for short) can be fitted to the data.
The maximum likelihood estimators of the Weibull shape and scale parameters are $2.101$  and $0.0122$, respectively. 
Based on the  Kolmogorov-Smirnov p-value of $0.6706$  obtained using these MLEs, the  Weibull  model  appears to provide a good fit to the data. 

Alternatively, the  gamma distribution may also be fitted to the data.
The corresponding  Kolmogorov-Smirnov p-value, obtained using the MLEs  $\alpha=4.0255$  and  $\lambda=0.0557$  is $ 0.878$.
Therefore, we consider  the two fitted distributions $Y_1\sim W(2.101, 0.0122)$ and $Y_2\sim G(4.0255,0.0557)$.
Under the data and  these models, the density function is estimated 
via a kernel estimator using  the density function in R.

The Kullback-Leibler divergences and varinaccuracies for the Weibull and gamma models are given, respectively by
$K(X, Y_1)=0.008416883$ ,\,\,  $ K(X, Y_2)=0.016953852$ ,\,\,     $\VarK(X, Y_1)=0.04069680$ and  $\VarK(X, Y_2)= 0.06079601$.
Since $Y_1$ yields both a smaller divergence and a smaller variance, these measures suggest that  $Y_1$ provides a better fit than $Y_2$.

Next, the discrimination measures are computed as
\begin{align}
J(X|Y_1)&= -6.547826e^{-05},\, \,     J(X|Y_2)=  -2.652552e^{-04},    \nonumber
\\   \VarJ(X|Y_1)&= 2.111388e^{-07},\,  \   \VarJ(X|Y_2)=2.102750e^{-07},
\end{align}
These results demonstrate that the distribution $Y_2$ provides a better fit to the data, since $J(X|Y_1)\leq J(X|Y_2)$ and $\VarJ(X|Y_1)\leq \VarJ(X|Y_2)$.
%%%%%%%%%%%%%%%%%%%%%%%%%%%%%%%%%%%%%%%%%
\begin{figure}
\begin{center}
\includegraphics[width=7cm, height=7cm]{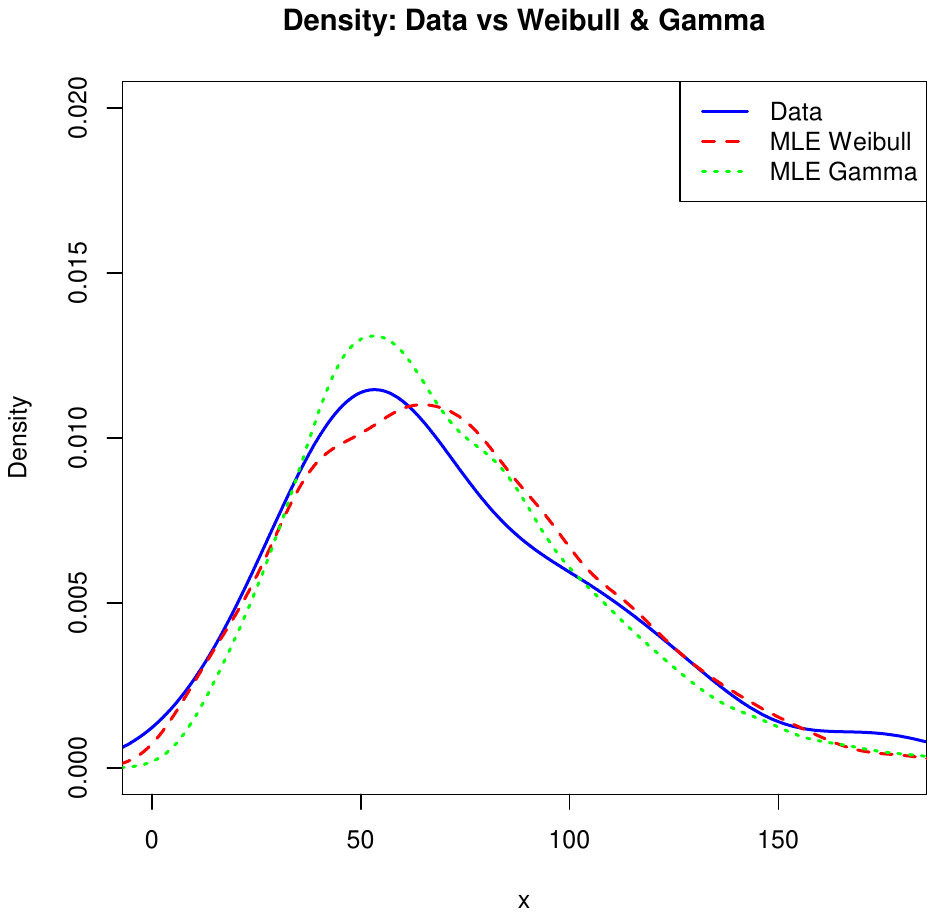}
\caption{\footnotesize{Plot of PDF's of $X$ , $Y_1$ and $Y_2$ for Example \ref{exp1}.}}\label{Fig1}
\end{center}
\end{figure}

In Figure \ref{Fig1}, the plot of the estimated PDF of the data and the PDFs $g_1$, $g_2$ of $Y_1$ and $Y_2$ can be seen.
Anderson-Darling (AD) statistic for $Y_1$ is $0.9146$, whereas  
for $Y_2$ it is $0.9857$, indicating that the gamma distribution fits the data better than the Weibull distribution. 

To provide a more comprehensive comparison, several criteria were used, including the Akaike information criterion (AIC) and its corrected version (CAIC) (see \cite{Hurvich}), the Bayesian information criterion (BIC), and the Hannan-Quinn information criterion (HQIC) (see \cite{Hann}), with their values shown in the  table below.    
%%%%%%%%%%%%%%%%%%%%%%%%%%%%%%%%%%%%%%%%%%%%%%%%%%%%%%%%%%%%%
\begin{table}[h]
\caption{AIC, CAIC, BIC, and HQIC for the two fitted models.}
\footnotesize
\centering
\begin{tabular}{lccccc}\label{Tb1}
Model &AIC&CAIC&BIC&HQIC\\
\hline
  Weibull& 231.3839&  231.9839
& 233.6549 &231.9551
\\
 Gamma&230.0596&230.6596&232.3306 &230.6308\\
\hline
\end{tabular}
\end{table}
From Table \ref{Tb1}, we conclude that the gamma distribution is superior, as it yields smaller values of AIC, CAIC, BIC, and HQIC. 
Therefore, the discrimination-based criteria appear to provide a more accurate identification of the better model compared with Kullback-Leibler measures.
\end{example}
%%%%%%%%%%%%%%%%%%%%%%%%%
\section{Conclusion}
In this paper, moment  generating functions associated with $J(X)$ and $J(X|Y)$ were derived, and the new measures  $\VarJ^w(X, Y)$ and $\VarJ(X, Y)$ based on varextropy were introduced. Several theoretical results for the proposed measures were obtained, including conditions under which  $\VarJ(X, Y)$ vanishes and lower bounds for it.
In addition, a discrimination measure based on varextropy  was introduced and studied for comparing probability distributions. 
 The numerical examples based on real lifetime data indicate that the proposed criteria can serve as useful tools for goodness-of-fit  assessment and model selection problems.  In particular, the proposed discrimination criterion provides direct and interpretable comparisons 
 between competing fitted models. Comparisons with Kullback-Leibler-based measures and standard goodness-of-fit procedures further support  the usefulness of the proposed approach. 
% References
%%%%%%%%%%%%%%%%%%%%%%%%%%%%%%%%%%%%%%%%%%%
\bibliographystyle{amsplain}
%%%%%%%%%%%%%%%%%%%%%%%%%%%%%%%%%%%%%%%%%%%
% Please cite your relevant papers but at most total 5 papers/books.
%%%%%%%%%%%%%%%%%%%%%%%%%%%%%%%%%%%%%%%%%%%

\end{document}